\documentclass{amsart}
\usepackage[all]{xy}
\usepackage{color}
\usepackage{amsthm}
\usepackage{amssymb}
\usepackage[colorlinks=true]{hyperref}%to generate  P D F  with links using latex

\numberwithin{equation}{subsection}

\newtheorem{theorem}{Theorem}[section]
\newtheorem*{theorem*}{Theorem}
\newtheorem{lemma}[theorem]{Lemma}
\newtheorem{proposition}[theorem]{Proposition}

\newtheorem*{corollary*}{Corollary}

\theoremstyle{remark}
\newtheorem{definition}[theorem]{Definition}
\theoremstyle{remark}

\theoremstyle{remark}

\theoremstyle{remark}

\DeclareMathOperator{\id}{id}

\newcommand{\Ho}{\mathsf{Ho}}% Homotopy category

\newcommand{\too}{\longrightarrow}% long rightarrow
\newcommand{\dg}{\mathsf{dg}}% index dg etc.
\newcommand{\cA}{{\mathcal A}}
\newcommand{\cB}{{\mathcal B}}
\newcommand{\cC}{{\mathcal C}}
\newcommand{\cD}{{\mathcal D}}

\newcommand{\cO}{{\mathcal O}}

\newcommand{\cT}{{\mathcal T}}
\newcommand{\cU}{{\mathcal U}}

\newcommand{\bbA}{\mathbb{A}}

\newcommand{\bbK}{I\mspace{-6.mu}K}

\newcommand{\bbZ}{\mathbb{Z}}

\newcommand{\ie}{\textsl{i.e.}\ }

\newcommand{\Hmo}{\mathsf{Hmo}}% Morita homotopy theory
\newcommand{\perf}{\mathsf{perf}} % Dg category of perfect complexes
\newcommand{\rep}{\mathsf{rep}} % rep symbol
\newcommand{\dgcat}{\mathsf{dgcat}}% category of dg categories
\newcommand{\Spt}{\mathsf{Sp}}% Spectra
\newcommand{\internalcomment}[1]{}

\title[The fundamental theorem via three simple properties]{The fundamental theorem via\\ derived Morita invariance, localization, \\and $\bbA^1$-homotopy invariance}
\author{Gon{\c c}alo~Tabuada}

\address{Gon\c calo Tabuada, Departamento de Matematica, FCT-UNL, Quinta da Torre, 2829-516 Caparica,~Portugal }
\email{tabuada@fct.unl.pt}

\date{\today}


\thanks{Research partially supported by the Clay Mathematics Institute and by the FCT-Portugal grant {\tt PTDC/MAT/098317/2008}.}

\begin{document}
\begin{abstract}
We prove that every functor defined on dg categories, which is derived Morita invariant, localizing, and $\bbA^1$-homotopy invariant, satisfies the fundamental theorem. As an application, we recover in a unified and conceptual way, Weibel and Kassel's fundamental theorems in homotopy algebraic $K$-theory, and periodic cyclic homology, respectively.
\end{abstract}
\maketitle

\vskip-\baselineskip
\vskip-\baselineskip
%\tableofcontents

\section{Statement of results}
A {\em differential graded (=dg) category}, over a fixed commutative base ring $k$, is a category enriched over 
cochain complexes of $k$-modules (morphisms sets are such complexes) in such a way that composition fulfills the Leibniz rule\,: $d(f\circ g)=(df)\circ g+(-1)^{\textrm{deg}(f)}f\circ(dg)$.
Dg categories enhance and solve many of the technical problems inherent to triangulated categories; see Keller's ICM address~\cite{ICM}. In non-commutative algebraic geometry in the sense of Bondal, Drinfeld, Kaledin,  Kontsevich, Van den Bergh~\cite{BB,Drinfeld,Chitalk,Kaledin,ENS,IAS,finMot} they are considered as dg-enhancements of (bounded) derived categories of quasi-coherent sheaves on a hypothetic non-commutative space. 

Given a dg algebra $A$, we will denote by $\underline{A}$ the associated dg category with one object and dg algebra of endomorphisms $A$. Let $E: \dgcat \to \cT$ be a functor, defined on the category of dg categories, and with values in an arbitrary triangulated category~\cite{Neeman}. The functor $E$ is called\,:
\begin{itemize}
\item[(i)] {\em Derived Morita invariant} if it inverts the dg functors $\cA \to \cB$ which induce an equivalence $\cD(\cA) \stackrel{\sim}{\to}\cD(\cB)$ between the associated derived categories; see \cite[\S3]{ICM}.
\item[(ii)] {\em Localizing} if it sends exact sequences of dg categories (see \cite[\S4.6]{ICM})  
to distinguished triangles
\begin{eqnarray*}
0 \to \cA \to \cB \to \cC \to 0 &\mapsto&
E(\cA) \to E(\cB) \to E(\cC) \to \Sigma(E(\cA))\,.
\end{eqnarray*}

\item[(iii)] {\em $\bbA^1$-homotopy invariant} if it inverts the dg functors $\cA \to \cA[t]$, where $\cA[t]$ denotes the tensor product of $\cA$ with $\underline{k[t]}$; see \cite[\S2.3]{ICM}.
\end{itemize}
Examples of functors satisfying conditions (i)-(iii) include homotopy algebraic $K$-theory $(KH)$ and periodic cyclic homology $(HP)$; see \S\ref{sec:KH}-\ref{sec:HP}.

Now, given any dg category $\cA$, let $\cA[t,t^{-1}]$ be the tensor product of $\cA$ with $\underline{k[t,t^{-1}]}$, where $k[t,t^{-1}]$ is the $k$-algebra of Laurent polynomials. We say that the functor $E$ {\em satisfies the fundamental theorem} if
\begin{equation*}
E(\cA[t,t^{-1}]) \simeq E(\cA) \oplus \Sigma(E(\cA))\,.
\end{equation*}
Our main result is the following\,:
\begin{theorem}\label{thm:fundamental}
Assume that $k$ is a field of characteristic zero. Then, every functor $E: \dgcat \to \cT$ satisfying conditions {\em (i)-(iii)} satisfies also the fundamental theorem.
\end{theorem}
By taking $E=KH$, resp. $E=HP$, in Theorem~\ref{thm:fundamental} we obtain the fundamental theorem in homotopy algebraic $K$-theory, resp. in periodic cyclic homology\,:
\begin{eqnarray}
KH(\cA[t,t^{-1}]) &\simeq& KH(\cA) \oplus \Sigma(KH(\cA))  \label{eq:KH} \\
HP(\cA[t,t^{-1}]) &\simeq& HP(\cA) \oplus \Sigma(HP(\cA))   \label{eq:HP}\,.
\end{eqnarray}
Given a quasi-compact and quasi-separated $k$-scheme $X$, it is well-known that the derived category $\cD_\perf(X)$ of perfect complexes of $\cO_X$-modules admits a natural dg-enhancement $\cD_\perf^\dg(X)$; see Lunts-Orlov~\cite{LO} or \cite[Example~4.5]{CT1}. Moreover, the homotopy algebraic $K$-theory and the periodic cyclic homology of the scheme $X$ can be recovered from the dg category $\cD_\perf^\dg(X)$ by applying to it the corresponding functor; see Propositions~\ref{prop:agreementKH} and \ref{prop:agreementHP}. Therefore, when $\cA=\cD_\perf^\dg(X)$, the above isomorphisms \eqref{eq:KH}-\eqref{eq:HP} reduce to\,:
\begin{eqnarray}
KH(X[t,t^{-1}]) &\simeq& KH(X) \oplus \Sigma(KH(X))  \label{eq:KH-schemes} \\
HP(X[t,t^{-1}]) &\simeq& HP(X) \oplus \Sigma(HP(X))   \label{eq:HP-schemes}\,,
\end{eqnarray}
where $X[t,t^{-1}]$ denotes the product (over $\mathrm{Spec}(k)$) of $X$ with $\mathrm{Spec}(k[t,t^{-1}])$.

Isomorphism \eqref{eq:KH}, resp. \eqref{eq:HP}, was originally established by Weibel, resp. by Kassel, in the particular case where $\cA$ is an ordinary $k$-algebra; see \cite[Theorem~1.2]{Weibel-KH} and \cite[Corollary~3.12]{Kassel}. Weibel also established isomorphism \eqref{eq:KH-schemes}; see \cite[Theorem~6.11]{Weibel-KH}. The strategies to construct the aforementioned isomorphisms are not only very different but also quite involved. Weibel's proof makes use, in an essential way, of Bass-Quillen's fundamental theorem in algebraic $K$-theory while Kassel's proof makes use of the monoidal properties of periodic cyclic homology as well as of its relation with Grothendieck's crystalline cohomology~\cite{Grot}.

In contrast, we obtain a unified and conceptual proof that applies to isomorphisms \eqref{eq:KH}-\eqref{eq:HP-schemes} since each one of these can be obtained by a simple instantiation of the general Theorem~\ref{thm:fundamental}. Moreover, to the best of the author's knowledge, isomorphism \eqref{eq:HP-schemes} is new in the literature. Theorem~\ref{thm:fundamental} uncovers in a direct and elegant way the three key conceptual ingredients that underlie the fundamental theorem. The strategy of its proof (see \S\ref{sec:proof}) differs radically from the one used by Weibel or by Kassel. It is based on a careful study of an explicit exact sequence, making use of an additive category, of ``motivic nature''\footnote{See Kontsevich~\cite[\S4.1.3]{finMot} for its precise relation with pure and mixed motives.}, developed by the author in \cite{IMRN}. Due to its generality and simplicity, we believe that Theorem~\ref{thm:fundamental} will soon be part of the toolkit of any mathematician whose research comes across the above conditions (i)-(iii).

\section{Homotopy algebraic $K$-theory of dg categories}\label{sec:KH}
In this section we introduce the homotopy algebraic $K$-theory of dg categories. Recall from Schlichting~\cite[\S12.1]{Negative} and \cite[Notation~2.38]{CT} the construction of the non-connective algebraic $K$-theory spectrum functor
\begin{equation}\label{eq:non-con}
\bbK: \dgcat \too \Ho(\Spt)\,,
\end{equation}
with values in the homotopy category of spectra~\cite{BF}. This functor, although derived Morita invariant and localizing, it is not  $\bbA^1$-homotopy invariant. In order to impose this latter property, we consider the simplicial $k$-algebra
\begin{eqnarray*}
\Delta_{\bullet}:&& \Delta_n:= k[t_0, \ldots, t_n]/ \big(\sum t_i -1\big) \qquad n \geq 0\,,
\end{eqnarray*}
with faces and degenerancies given by the formulae
\begin{eqnarray*}
\partial_r(t_j) := \left\{ \begin{array}{lcr}
t_j & \text{if} & j <r \\
0 & \text{if} & j =r \\
t_{j-1} & \text{if} & j > r \\
\end{array} \right.
&
&
\delta_r(t_j) := \left\{ \begin{array}{lcr}
t_j & \text{if} & j <r \\
t_j + t_{j+1} & \text{if} & j =r \\
t_{j+1} & \text{if} & j > r \\
\end{array} \right.\,.
\end{eqnarray*}
\begin{definition}
The {\em homotopy algebraic $K$-theory spectrum functor} is defined as\,:
\begin{eqnarray}\label{def:KH}
KH: \dgcat \too \Ho(\Spt) && \cA \mapsto \mathrm{hocolim}_n\, \bbK(\cA \otimes \underline{\Delta_n})\,.
\end{eqnarray}
\end{definition}
\begin{proposition}\label{prop:KH-A1}
The functor \eqref{def:KH} is $\bbA^1$-homotopy invariant.
\end{proposition}
\begin{proof}
Note first that we have natural morphisms of $k$-algebras $\iota: k \to k[t]$ and $\mathrm{ev}_0, \mathrm{ev}_1: k[t] \to k$ satisfying the relations $\mathrm{ev}_0 \circ \iota =\mathrm{ev}_1 \circ \iota=\id$. Given a dg category $\cA$, the dg functor $\cA \to \cA[t]$ of condition (iii) corresponds then to 
$$\cA \stackrel{\sim}{\too } \cA \otimes \underline{k} \stackrel{\id \otimes \iota}{\too} \cA\otimes \underline{k[t]}\,.$$
Therefore, in order to prove that the functor \eqref{def:KH} is $\bbA^1$-homotopy invariant we need to show that the induced morphism of spectra
$$ \mathrm{hocolim}_n\, \bbK(\cA \otimes \underline{\Delta_n}) \simeq \mathrm{hocolim}_n\, \bbK(\cA \otimes \underline{k} \otimes \underline{\Delta_n}) \too \mathrm{hocolim}_n\, \bbK(\cA \otimes \underline{k[t]}\otimes \underline{\Delta_n})$$
is an isomorphism in $\Ho(\Spt)$. This follows from the equality $\mathrm{ev}_0\circ \iota = \id$ and from the existence of a simplicial homotopy (see \cite[Proposition~2.4]{Weibel-KV}) between the identity and the simplicial map of $k$-algebras
$$ (k[t] \stackrel{\mathrm{ev}_0}{\too} k \stackrel{\iota}{\too} k[t]) \otimes \Delta_{\bullet}\,.$$
\end{proof}
The fact that the functor \eqref{def:KH} is derived Morita invariant and localizing follows from its own definition and from the corresponding properties of the functor \eqref{eq:non-con}.
\begin{proposition}\label{prop:agreementKH}
Let $\cA$ be a dg category.
\begin{itemize}
\item[(a)] When $\cA$ is derived Morita equivalent to $\underline{A}$, with $A$ an ordinary $k$-algebra, then $KH(\cA)$ is isomorphic to the homotopy algebraic $K$-theory spectrum of $A$ as defined by Weibel in \cite[Definition~1.1]{Weibel-KH}. 
\item[(b)]  When $\cA$ is derived Morita equivalent to $\cD_\perf^{\dg}(X)$, with $X$ a quasi-compact and quasi-separated $k$-scheme, then $KH(\cA)$ is isomorphic to the homotopy algebraic $K$-theory spectrum of $X$ as defined by Weibel in \cite[Definition~6.5]{Weibel-KH}.
\end{itemize}
\end{proposition}
\begin{proof}
Item (a) follows immediately from Weibel's definition~\cite[Definition~1.1]{Weibel-KH}, from the natural identification $ A \otimes \Delta_n \simeq \Delta_nA$ where $\Delta_nA$ is the coordinate ring $A[t_0, \dots, t_n] / \big(\sum t_i -1\big)A$ of ``standard $n$-simplexes'' over $A$, and from the fact that for every $k$-algebra $B$ the spectrum $\bbK(\underline{B})$ agrees with the non-connective algebraic $K$-theory spectrum of $B$. In what concerns item (b), we have the following equivalences
\begin{equation*}
\cD_\perf^\dg(X) \otimes \underline{\Delta_n} \simeq \cD_\perf^\dg(X \times \mathrm{Spec}(\Delta_n)) \qquad n \geq 0\,.
\end{equation*}
Therefore, in this particular case where $\cA=\cD_\perf^{\dg}(X)$, Definition~\eqref{def:KH} reduces to
\begin{equation}\label{eq:lasteq}
\mathrm{hocolim}_n\, \bbK(\cD_\perf^\dg(X \times \mathrm{Spec}(\Delta_n)))\,.
\end{equation}
The (non-connective) algebraic $K$-theory spectrum of the scheme $X \times \mathrm{Spec}(\Delta_n)$ (see \cite[\S6.5 and \S12.1]{Negative}) agrees with the (non-connective) algebraic $K$-theory spectrum of the dg category $\cD_\perf^{\dg}(X \times \mathrm{Spec}(\Delta_n))$; see \cite[Theorem~5.1]{ICM} and \cite[\S8 Theorem~5]{Negative}. Hence, the above homotopy colimit \eqref{eq:lasteq} can be identified with
\begin{equation*}
\mathrm{hocolim}_n\, \bbK(X \times \mathrm{Spec}(\Delta_n))\,.
\end{equation*}
Finally, since $X$ is quasi-compact and quasi-separated, we conclude that this latter spectrum is equivalent to the one defined by Weibel in \cite[Definition~6.5]{Weibel-KH} using \v{C}ech's cohomology descent; see Thomason-Trobaugh~\cite[\S9.11]{Thomason}.
\end{proof}
\section{Periodic cyclic homology of dg categories}\label{sec:HP}
In this section we recall the periodic cyclic homology of dg categories. Following Kassel~\cite[\S1]{Kassel}, a {\em mixed complex} $(M,b,B)$ is a $\bbZ$-graded $k$-module $(M_n)_{n \in \bbZ}$ endowed with a degree $+1$ endomorphism $b$ and a degree $-1$ endomorphism $B$ satisfying the relations $b^2=B^2=Bb +bB=0$. Equivalently, a mixed complex is a right dg module over the dg algebra $\Lambda:=k[\epsilon]/\epsilon^2$, where $\epsilon$ is of degree $-1$ and $d(\epsilon)=0$. Recall from \cite[Example~7.10]{CT1} the construction of the mixed complex functor
\begin{equation}\label{eq:mixed}
C: \dgcat \too \cD(\Lambda)\,,
\end{equation}
with values in the derived category of $\Lambda$. As explained by Kassel in~\cite[page~210]{Kassel} there is a well-defined functor sending a mixed complex $(M,b,B)$ to the $\bbZ/2$-graded cochain complex
$$
\prod_{\mathrm{n\,even}} \xymatrix{M_n \ar@<1ex>[r]^-{b+B} & \ar@<1ex>[l]^-{b+B}}  \prod_{\mathrm{n\,odd}} M_n \,.
$$
This functor preserves weak equivalences and so combined with \eqref{eq:mixed} gives rise to a composed functor
\begin{equation}\label{eq:HP1}
HP: \dgcat \stackrel{C}{\too} \cD(\Lambda) \too \cD_{\bbZ/2}(k)
\end{equation}
with values in the derived category of $\bbZ/2$-graded cochain complexes of $k$-modules, \ie the derived category of those $\bbZ$-graded complexes which are invariant under the $2$-fold shift functor. 
\newpage
\begin{proposition}
When $k$ is a field of characteristic zero, the functor \eqref{eq:HP1} is $\bbA^1$-homotopy invariant.
\end{proposition}
\begin{proof}
Kassel's property $(P)$ (see \cite[page~211]{Kassel}) is verified by the $k$-algebras $k$ and $k[t]$. Therefore, \cite[Theorem~3.10]{Kassel} furnish us the isomorphisms\,:
\begin{eqnarray*}
HP(\cA \otimes \underline{k}) \simeq HP(\cA) \otimes HP(\underline{k}) && HP(\cA \otimes \underline{k[t]}) \simeq HP(\cA) \otimes HP(\underline{k[t]})\,.
\end{eqnarray*}
This shows us that the functor \eqref{eq:HP1} is $\bbA^1$-homotopy invariant if and only if it inverts the dg functor $\underline{k} \to \underline{k[t]}$. Since $k$ is a field of characteristic zero, Kassel's homotopy invariance results (see \cite[Corollary 3.12 and (3.13)]{Kassel}) allow us to conclude that \eqref{eq:HP1} inverts the dg functor $\underline{k} \to \underline{k[t]}$ and so the proof is finished.
\end{proof}
The fact that \eqref{eq:HP1} is derived Morita invariant and localizing follows from the corresponding properties of the functor \eqref{eq:mixed}.
\begin{proposition}{(Keller~\cite[Theorem~5.2]{ICM})}\label{prop:agreementHP}
Let $\cA$ be a dg category.
\begin{itemize}
\item[(a)] When $\cA$ is derived Morita equivalent to $\underline{A}$, with $A$ an ordinary $k$-algebra, then $HP(\cA)$ is isomorphic to the periodic cyclic homology of $A$.
\item[(b)]  When $\cA$ is derived Morita equivalent to $\cD_\perf^{\dg}(X)$, with $X$ a quasi-compact and quasi-separated $k$-scheme, then $HP(\cA)$ is isomorphic to the periodic cyclic homology of $X$.
\end{itemize}
\end{proposition}
Recall from Feigin-Tsygan \cite{FT} that when $A$ is the coordinate $k$-algebra of an affine algebraic variety $V$ over a field $k$ of characteristic zero, $HP(A)$ is isomorphic to the crystalline cohomology of $V$ defined by Grothendieck~\cite{Grot} as a generalization of de Rham cohomology for singular varieties.
%--------------------------------------------------
\section{Proof of Theorem~\ref{thm:fundamental}}\label{sec:proof}
%--------------------------------------------------
We start by studying the behavior of the functor $E$ with respect to gradings.
\begin{lemma}\label{lem:key1}
Let $A=\bigoplus_{n=0}^\infty\, A_n$ be a non-negatively graded $k$-algebra and let $\cA$ be an arbitrary dg category. Then, the dg functor 
\begin{equation}\label{eq:inclusion}
\id \otimes i: \cA\otimes \underline{A_0} \to \cA\otimes \underline{A}
\end{equation}
becomes invertible after application of the functor $E$. 
\end{lemma}
\begin{proof}
Let us denote by 
\begin{equation}\label{eq:projection}
\id \otimes p: \cA \otimes \underline{A} \to \cA \otimes \underline{A_0}
\end{equation}
the naturally associated ``projection'' dg functor. The proof consists in showing that the dg functors \eqref{eq:inclusion} and \eqref{eq:projection} become inverse of each other after application of the functor $E$.
Recall from the proof of Proposition~\ref{prop:KH-A1} that we have natural morphisms of $k$-algebras $\iota:~k \to~k[t]$ and $\mathrm{ev}_0, \mathrm{ev}_1: k[t] \to k$ satisfying the relations $\mathrm{ev}_0 \circ \iota =\mathrm{ev}_1 \circ \iota=\id$.  On one hand we have the equality $(\id \otimes p) \circ (\id \otimes i) = \id$. On the other hand, we have the following commutative diagram in $\dgcat$
\begin{equation*}
\xymatrix@C=3em@R=1.5em{
&&& \cA \otimes \underline{A} \\
&&& \cA \otimes \underline{A} \otimes \underline{k} \ar[u]^{\sim}\\
\cA \otimes \underline{A} \ar@/^2.0pc/@{=}[uurrr] \ar@/_2.0pc/[ddrrr]_-{\id \otimes (i\circ p)} \ar[rrr]^-{\id \otimes H} &&& \cA \otimes \underline{A}\otimes \underline{k[t]} \ar[u]_{\id \otimes \id \otimes \mathrm{ev}_1} \ar[d]^{\id \otimes \id \otimes \mathrm{ev}_0}\\
&&& \cA \otimes \underline{A} \otimes \underline{k} \ar[d]^{\sim} \\
&&& \cA \otimes \underline{A}\,,
}
\end{equation*}
where $H$ is the dg functor that sends an homogeneous element $a_n \in A$ of degree $n$ to $a_n\otimes t^n$. Since $\id \otimes (i \circ p) = (\id \otimes i) \circ (\id \otimes p)$ it suffices to show that the dg functors $\id\otimes \id \otimes \mathrm{ev}_0$ and $\id\otimes \id \otimes \mathrm{ev}_1$ become invertible after application of the functor $E$. As a consequence of the relations
\begin{eqnarray*}
(\id \otimes \id \otimes \mathrm{ev}_0) \circ (\id \otimes \id \otimes \iota) = \id && (\id \otimes \id \otimes \mathrm{ev}_1) \circ (\id \otimes \id \otimes \iota) = \id
\end{eqnarray*}
it is enough to show this latter claim for the dg functor $\id \otimes \id \otimes \iota$. Finally, this follows from the fact that the composition
$$\cA \otimes \underline{A} \stackrel{\sim}{\too}  \cA \otimes \underline{A}\otimes \underline{k} \stackrel{\id \otimes \id \otimes \iota}{\too} \cA \otimes \underline{A} \otimes \underline{k[t]}$$
corresponds to the dg functor of condition (iii), with $\cA$ replaced by $\cA\otimes \underline{A}$.
\end{proof}
We now study the localization of the $k$-algebra $k[t]$ at the multiplicative set $\{1,t,t^2,t^3, \ldots\}$. Following Keller~\cite[\S4.1]{Keller} this localization gives rise to an exact sequence of triangulated categories
\begin{equation}\label{eq:ses}
0 \too \cD(k[\xi]/\xi^2) \stackrel{-\otimes X}{\too} \cD(k[t]) \too \cD(k[t,t^{-1}]) \too 0\,.
\end{equation}
Some explanations are in order: $k[\xi]/\xi^2$ is the {\em graded} $k$-algebra of dual numbers, where $\xi$ is of degree $1$, and $X$ is a $k[\xi]/\xi^2\text{-}k[t]$-bimodule $X$ whose restriction to $k[t]$ is the following projective resolution
\begin{equation}\label{eq:complex}
{\bf P}:=(0 \too k[t] \stackrel{\cdot t}{\too} k[t] \too 0)
\end{equation}
of the trivial $k[t]$-module $k$; see~\cite[\S4.1 Example~a)]{Keller} for further details.

Now, recall from \cite[Theorem~5.3]{IMRN} that the category $\dgcat$ carries a Quillen model structure~\cite{Quillen} whose weak equivalences are the {\em derived Morita equivalences}, \ie the dg functors $\cA \to \cB$ which induce an equivalence $\cD(\cA) \stackrel{\sim}{\too}\cD(\cB)$ between the associated derived categories. The homotopy category hence obtained will be denoted by $\Hmo$. Given dg categories $\cA$ and $\cB$, let $\rep(\cA,\cB)$ be the derived category of those $\cA\text{-}\cB$-bimodules which are perfect as $\cB$-modules for every object of $\cA$. The set of morphisms in $\Hmo$ from $\cA$ to $\cB$ is given by the isomorphism classes of objects in this triangulated category $\rep(\cA, \cB)$; see \cite[Corollary~5.10]{IMRN}. Moreover, composition is induced by the (derived) tensor product of bimodules.

Since the above complex \eqref{eq:complex} of $k[t]$-modules is clearly perfect, the $\underline{k[\xi]/\xi^2}\text{-}\underline{k[t]}$-bimodule $X$ belongs to $\rep(\underline{k[\xi]/\xi^2}, \underline{k[t]})$. By combining this fact with \eqref{eq:ses} we obtain an exact sequence of dg categories
\begin{equation*}
0 \too \underline{k[\xi]/\xi^2} \stackrel{X}{\too} \underline{k[t]} \too \underline{k[t,t^{-1}]} \too 0
\end{equation*}
in $\Hmo$. Given an arbitrary dg category $\cA$, Drinfeld's Proposition~\cite[Proposition~1.6.3]{Drinfeld} furnish us then the following exact sequence of dg categories
\begin{equation}\label{eq:ex2}
0 \too \cA \otimes \underline{k[\xi]/\xi^2} \stackrel{\id \otimes X}{\too} \cA \otimes \underline{k[t]} \too \cA \otimes \underline{k[t,t^{-1}]} \too 0\,.
\end{equation}
We now study the behavior of the functor $E$ with respect to the exact sequence \eqref{eq:ex2}. Notice that since $E$ is derived Morita invariant, it descends to the homotopy category $\Hmo$, \ie the localization of $\dgcat$ with respect to the class of derived Morita equivalences. Moreover, since $E$ is also localizing, the above exact sequence \eqref{eq:ex2} is sent to the distinguished triangle
\begin{equation}\label{eq:triang}
E(\cA \otimes \underline{k[\xi]/\xi^2}) \stackrel{E(\id \otimes X)}{\too} E(\cA[t]) \too E(\cA[t,t^{-1}])  \too \Sigma(E(\cA \otimes \underline{k[\xi]/\xi^2})) \,.
\end{equation}
\begin{lemma}\label{lem:key2}
The morphism $E(\id \otimes X)$, in the above triangle \eqref{eq:triang}, is zero.
\end{lemma}
\begin{proof}
We start by briefly recalling from \cite{IMRN} the construction, and the universal property, of the additive category $\Hmo_0$. Its objects are all dg categories. Given dg categories $\cA$ and $\cB$, the abelian group of morphisms in $\Hmo_0$ from $\cA$ to $\cB$ is the Grothendieck group $K_0\rep(\cA, \cB)$ of the triangulated category $\rep(\cA, \cB)$. Composition is induced, as in $\Hmo$, by the (derived) tensor product of bimodules; see \cite[\S6]{IMRN}. Notice that we have a natural (composed) functor
\begin{equation}\label{eq:univ}
\cU: \dgcat \too \Hmo \too \Hmo_0\,.
\end{equation}
It maps a dg functor $\cA \to \cB$ to the corresponding $\cA\text{-}\cB$-bimodule in $\rep(\cA, \cB)$, and a $\cA\text{-}\cB$-bimodule in $\rep(\cA, \cB)$ to the corresponding class in the Grothendieck group $K_0\rep(\cA, \cB)$. In \cite[Theorem~6.3]{IMRN} it was proved that \eqref{eq:univ} is the {\em universal} functor, with values in an additive category, which is simultaneously derived Morita invariant (see condition (i)) and additive. By additive we mean that it sends split exact sequences of dg categories (see \cite[\S13]{Duke}) to direct sums
\begin{eqnarray}\label{eq:additivity}
\xymatrix@C=1.5em@R=1.0em{
0 \ar[r] &  \cA \ar[r]  & \cB \ar[r]  \ar@/_0.5pc/[l] & \cC \ar[r] \ar@/_0.5pc/[l] &  0
} &\mapsto&
\cU(\cA) \oplus \cU(\cC) \stackrel{\sim}{\to} \cU(\cB)\,,
\end{eqnarray}

%0 \to \cA \stackrel{\leftarrow}{\to} \cB  \stackrel{\leftarrow}{\to} \cC \to 0 \qquad \mapsto \qquad \cU(\cA)\oplus \cU(\cC) \stackrel{\sim}{\to} \cU(\cB)\,.
%\end{equation}
Moreover, the symmetric monoidal structure on $\dgcat$ (given by the tensor product) was extended to $\Hmo$ and to $\Hmo_0$, making the functor \eqref{eq:univ} symmetric monoidal.

We now study the morphism $E(\id \otimes X)$ making use of the above functor \eqref{eq:univ}. Observe that we have the following commutative diagram in $\Hmo$
\begin{equation}\label{eq:diagram1}
\xymatrix{
\underline{k[\xi]/\xi^2} \ar[rr]^X && \underline{k[t]} \\
\underline{k} \ar[u]^{i} \ar[urr]_{\bf P} && \,.
}
\end{equation}
The description given in \eqref{eq:complex} show us that we can identify ${\bf P}$ with the mapping cone complex (see \cite[\S1.5.1]{Weibel}) of the morphism
$$ \cdot t: k[t] \too k[t]$$
between complexes of $k[t]$-modules concentrated in degree zero. Therefore, in the derived category of perfect complexes of $k[t]$-modules, which is naturally equivalent to $\rep(\underline{k},\underline{k[t]})$, we obtain a triangle
$$ k[t] \stackrel{\cdot t}{\too} k[t] \too {\bf P} \too \Sigma (k[t])\,.$$
This implies that the class of the $\underline{k}\text{-}\underline{k[t]}$-bimodule ${\bf P}$ in the Grothendieck group $K_0 \rep(\underline{k}, \underline{k[t]})$ is trivial. In terms of the additive category $\Hmo_0$, this means that ${\bf P}$ becomes the zero morphism in $\Hmo_0$. By tensoring \eqref{eq:diagram1} with an arbitrary dg category $\cA$ we obtain the following commutative diagram in $\Hmo$
\begin{equation}\label{eq:diagram2}
\xymatrix{
\cA\otimes \underline{k[\xi]/\xi^2} \ar[rr]^{\id \otimes X} && \cA \otimes \underline{k[t]} \\
\cA \otimes \underline{k} \ar[u]^{\id \otimes i} \ar[urr]_{\id \otimes {\bf P}} && \,.
}
\end{equation}
Since the morphism ${\bf P}$ becomes zero in the additive category $\Hmo_0$ and the functor \eqref{eq:univ} is symmetric monoidal, the morphism $\id \otimes {\bf P}$ in diagram \eqref{eq:diagram2} also becomes zero in $\Hmo_0$. Now, note that every triangulated category is in particular an additive category, and that the notion of localizing functor of condition (ii) implies the above notion~\eqref{eq:additivity} of additive functor. Therefore, since $E$ is simultaneously derived Morita invariant and localizing, the universal property of \eqref{eq:univ} allow us to conclude that $E$ factors through $\Hmo_0$. We obtain then a commutative diagram
$$
\xymatrix{
\dgcat\ar[d]_{\cU} \ar[rr]^E && \cT \\
\Hmo_0 \ar[urr]_{\overline{E}} && \,,
}
$$
with $\overline{E}$ an additive functor. These arguments imply that the morphism $E(\id \otimes {\bf P})$ is zero. Finally, since the $k$-algebra $k[\xi]/\xi^2$ admits a natural non-negatively graded graduation with $(k[\xi]/\xi^2)_0=k$, the above commutative diagram \eqref{eq:diagram2} combined with Lemma~\ref{lem:key1} allow us to conclude that the morphism $E(\id \otimes X)$ is zero. 
\end{proof}
We are now ready to finish the proof of Theorem~\ref{thm:fundamental}. Thanks to Lemma~\ref{lem:key2}, the morphism $E(\id \otimes X)$ is zero and so by \cite[Corollary~1.2.7]{Neeman} the triangle \eqref{eq:triang} splits and we obtain
\begin{equation}\label{eq:isom2}
E(\cA[t,t^{-1}]) \simeq E(\cA \otimes \underline{k[t]}) \oplus \Sigma(E(\cA \otimes \underline{k[\xi]/\xi^2}))\,.
\end{equation}
Notice that the $k$-algebras $k[t]$ and $k[\xi]/\xi^2$ admit natural non-negatively graded graduations with $(k[t])_0=k$ and $(k[\xi]/\xi^2)_0=k$. Therefore, since $\cA \otimes \underline{k} \simeq \cA$, Lemma~\ref{lem:key1} combined with isomorphism \eqref{eq:isom2} furnishes finally us the desired isomorphism
$$E(\cA[t,t^{-1}]) \simeq E(\cA) \oplus \Sigma(E(\cA))\,.$$
This concludes the proof.
\medbreak
\noindent\textbf{Acknowledgments:} The author is very grateful to Paul Balmer for stimulating questions, to Christian Haesemeyer for important discussions, to Bernhard Keller for valuable conversations and for his enthusiasm regarding these results, and to the anonymous referee for his comments and suggestions which greatly improved the article. The author would also like to thank the Department of Mathematics at UCLA for its hospitality and excellent working conditions, where this work was initiated.

\end{document}